\documentclass[12pt,letterpaper]{amsart}
\usepackage{amssymb,amsmath,amscd,eucal,epsfig}
\def\<{\langle}                     %added
\def\>{\rangle}                     %added
\newcommand{\ben}{\begin{enumerate}}
\newcommand{\een}{\end{enumerate}}

\theoremstyle{plain}
\newtheorem{theorem}{Theorem}[section]
\newtheorem{proposition}{Proposition}[section]

\newtheorem{corollary}{Corollary}[section]
\newtheorem{remark}{Remark}[section]

\theoremstyle{definition}
\newtheorem{definition}{Definition}[section]

\numberwithin{equation}{section}
\begin{document}

\begin{center}
{\textbf{{Direct Product of Picture Fuzzy Subgroups}}}\ \\ \ \\

{Taiwo O. Sangodapo},\\
Department of Mathematics,\\ University of Ibadan, Ibadan, Nigeria\\ to.sangodapo@ui.edu.ng,
toewuola77@gmail.com
\end{center}
\ \\ \  \\

\begin{center} {\textbf{Abstract}} \end{center}
In this paper, the concept of a picture fuzzy subgroup of a group is studied, and the notion of the direct product of picture fuzzy subgroups is introduced. Several characterisations of the direct product of picture fuzzy subgroups are established using the $(r, s, t)$-cut sets of picture fuzzy sets. 
\; \\ \; \\ \; \\

\noindent \textbf{MSC Classification:} 03E72, 08A72
\ \\ \ \\

\noindent {\textbf{Keywords}}: Picture Fuzzy Set, Picture Fuzzy Subgroup, Direct Product, Cut set, Picture Fuzzy Normal Subgroup, Coset

\section{Introduction}
Fuzzy set theory was introduced by Zadeh, 1965 \cite{z} as an extension of classical set theory. Zadeh's work was generalised to intuitionistic fuzzy set by Atanassov, 1984 \cite{a1}. The theories of fuzzy sets and intuitionistic fuzzy sets have been studied and applied to other areas see \cite{ta1, ta2, ta3} for details.  In 2013, Cuong and Kreinovich \cite{ck} generalised both fuzzy set and intuitionistic fuzzy set to picture fuzzy set. In other theories before the adventure of picture fuzzy set, the degree of neutrality was not incorporated. This important concept can be seen in the voting system where human beings are of the opinions to vote for, vote against, abstain from voting and refusal of voting, and also in medical diagnosis. The theory of picture fuzzy set has been widely studied and applied in various real life problems see \cite{c4, ch, ckn, dg, to, t1, t2, t3, t4, s1, s2} for details.

One important direction in this development is the study of picture fuzzy subgroups, which generalise classical subgroups and fuzzy subgroups of a group introduced by Dogra and Pal \cite{dp} whereby extending both Rosenfield's and Biswas' works. Sangodapo and Onasanya \cite{so} contributed to the work of Dogra and Pal \cite{dp} to establish some characteristics of PFSG via $(r, s, t)$-cut sets of a PFS.

In this work, contribution was made to the work of Sangodapo and Onasanya \cite{so} to obtain some characterisations of the direct product of picture fuzzy subgroups via $(r, s, t)$-cut sets.

\ \\

\section{Preliminaries}
This section gives the basic definitions and existing results relating to picture fuzzy subgroups.

\begin{definition} \cite{ck}
A picture fuzzy set Q of $Y$ is defined as $$Q = \lbrace (y, \sigma_{Q}(y), \tau_{Q}(y), \gamma_{Q}(y))| y \in Y \rbrace,$$ where the functions $$\sigma_{Q}: Y \rightarrow [0, 1],~ \tau_{Q}: Y \rightarrow [0, 1]~ \text{and}~ \gamma_{Q}: Y \rightarrow [0, 1]$$
are called the positive, neutral and negative membership degrees of $y \in Q$, respectively, and $\sigma_{Q}, \tau_{Q}, \gamma_{Q}$ satisfy $$0 \leq \sigma_{Q}(y) + \tau_{Q}(y) + \gamma_{Q}(y) \leq 1,~ \forall y \in Y.$$ For each $y \in Y$, $S_{Q}(y) = 1 - (\sigma_{Q}(y) - \tau_{Q}(y) - \gamma_{Q}(y))$ is called the refusal membership degree of $y \in Q$.
\end{definition}

\begin{definition}\cite{dp}
Let $(G, \ast)$ be a crisp group and $Q = \lbrace (y, \sigma_{Q}(y), \tau_Q(y), \eta_Q(y))~|~ y \in G \rbrace$ be a PFS in $G$. Then, $Q$ is called picture fuzzy subgroup of $G$ (PFSG) if\\ (i)~ $\sigma_Q(a \ast b) \geq \sigma_Q(a) \wedge \sigma_Q(b),~ \tau_Q(a \ast b) \geq \tau_Q(a) \wedge \tau_Q(b),~ \eta_Q(a \ast b) \leq \eta_Q(a) \vee \eta_Q(b) $\\ (ii) $\sigma_Q(a^{-1}) \geq \sigma_Q(a),~ \tau_Q(a^{-1}) \geq \tau_Q(a),~ \eta_Q(a^{-1}) \leq \eta_Q(a)$ for all $a, b \in G$.\\ Notice that $a^{-1}$ is the inverse of $a \in G$,\\ or equivalently, $Q$ is a PFSG of $G$ if and only if\\ $\sigma_Q(a \ast b^{-1}) \geq \sigma_Q(a) \wedge \sigma_Q(b),~ \tau_Q(a \ast b^{-1}) \geq \tau_Q(a) \wedge \tau_Q(b),~ \eta_Q(a \ast b^{-1}) \leq \eta_Q(a) \vee \eta_Q(b) .$
\end{definition}

\begin{definition}\cite{dp}
Let $(G, \ast)$ be a crisp group and $Q = (\sigma_{Q}, \tau_Q, \eta_Q)$ be a PFSG of $G$. Then, $Q$ is called picture fuzzy normal subgroup of $G,$ denoted by PFNSG if $$\sigma_{Qa}(b) = \sigma_{aQ}(b),~ \tau_{Qa}(b) = \tau_{aQ}(b),~ \eta_{Qa}(b) = \eta_{aQ}(b)$$ for all $a,~ b \in G$.
\end{definition}
\noindent The above definition can be redefined as;
\begin{definition}
A PFSG $Q = (\sigma_{Q}, \tau_Q, \eta_Q)$ of a group $G$ is said to be a PFNSG of $G$ if $\sigma_Q(ab) = \sigma_Q(ba),~ \tau_Q(ab) = \tau_Q(ba)~ \text{and}~ \eta_Q(ab) = \eta_Q(ba)~\forall~ a, b \in G,$ or equivalently, $Q$ a PFSG of $G$ is said to be normal if and only if $$\sigma_Q(b^{-1}ab) = \sigma_Q(a),~ \sigma_Q(b^{-1}ab) = \sigma_Q(a)~ \text{and}~ \eta_Q(b^{-1}ab) = \eta_Q(a)~ \forall~ a, b \in G.$$
\end{definition}

\begin{definition} \cite{dp}
Let $Q  = \lbrace (x, \sigma_{Q}, \tau_{Q}, \eta_{Q})| a \in X \rbrace$ be PFS over the universe $X$. Then, $(r,s,t)$-cut set of $Q$ is the crisp set in $Q$, denoted by $C_{r, s, t} (Q)$ and is defined by $$C_{r, s, t} (Q) = \left\lbrace a \in X | \sigma_{Q}(a) \geq r,~\tau_{Q}(a) \geq s, \eta_{Q}(a) \leq t  \right\rbrace$$ $r, s, t \in [0, 1]$ with the condition  $0 \leq r + s + t \leq 1$
\end{definition}

\begin{theorem}\cite{dp}\label{1}
Let $(G, \ast)$ be a crisp group and $Q = (\sigma_{Q}, \tau_Q, \eta_Q)$ be a PFSG of $G$. Then, $Q$ is a PFSG (PFNSG) if and only if $C_{r, s, t} (Q)$ is a crisp subgroup (normal) of $G$.
\end{theorem}

\begin{definition} \cite{dp}
Let $(G, \ast)$ be a crisp group and $Q = (\sigma_{Q}, \tau_Q, \eta_Q)$ be a PFSG of $G$. Then, for $a \in G$ the picture fuzzy left coset of $Q \in G$ is the PFS $aQ = (\sigma_{aQ},~ \tau_{aQ},~ \eta_{aQ})$ defined by $$\sigma_{aQ}(u) = \sigma_{Q} (a^{-1} \ast u),~ \tau_{aQ}(u) = \tau_{Q} (a^{-1} \ast u)~ \text{and}~  \eta_{aQ}(u) = \eta_{Q} (a^{-1} \ast u)$$ for all $ u \in G$.
\end{definition}

\begin{definition}\cite{dp}
Let $(G, \ast)$ be a crisp group and $Q = (\sigma_{Q}, \tau_Q, \eta_Q)$ be a PFSG of $G$. Then, for $a \in G$ the picture fuzzy right coset of $Q \in G$ is the PFS $Qa = (\sigma_{Qa},~ \tau_{Qa},~ \eta_{Qa})$ defined by $$\sigma_{Qa}(u) = \sigma_{Q} (u \ast a^{-1}),~ \tau_{Qa}(y) = \tau_{Q} (u \ast a^{-1})~ \text{and}~ \eta_{Qa}(u) = \eta_{Q} ( u \ast a^{-1})$$ for all $u \in G$.
\end{definition}

\begin{definition}\cite{dp}
Let $(G, \ast)$ be a crisp group and $Q = (\sigma_{Q}, \tau_Q, \eta_Q)$ be a PFSG of $G$. Then, $Q$ is called picture fuzzy normal subgroup (PFNSG) of $G$ if $$\sigma_{Qa}(y) = \sigma_{aQ}(y),~ \tau_{Qa}(y) = \tau_{aQ}(y),~ \eta_{Qa}(y) = \eta_{aQ}(y)$$ for all $a,~ y \in G$.
\end{definition}

\begin{proposition} \cite{so}
Let $Q$ be a PFSG of $G$ and $a$ be any fixed element of $G.$ Then, \ \\ \begin{itemize}
\item [(i.)] $a \ast C_{r, s, t} (Q) = C_{r, s, t} (a \ast Q).$ \ \\ \item [(ii.)] $C_{r, s, t} (Q) \ast a = C_{r, s, t} (Q \ast a),~ \forall~ r, s, t \ [0, 1] $ with $r + s + t \leq 1.$
\end{itemize}
\end{proposition}

\begin{definition}\cite{t4}
Let $Y_1$ and $Y_2$ be two nonempty sets and $f:Y_1 \rightarrow Y_2$ be a mapping. Let $P$ and $Q$ be two PFSs of $Y_1$ and $Y_2,$ respectively. Then, the image of $P$ under $f$ denoted by $f(P)$ is defined as $$f(P)(y_2) = (\sigma_{f(P)}(y_2),~ \tau_{f(P)}(y_2),~ \eta_{f(P)}(y_2)),$$ where
$$\sigma_{f(P)}(y_2) = \left\{\begin{array}{ll} \vee\{\sigma_P(y_1): y_1 \in  f^{-1}(y_2)\} \\
\ \\

0,\; otherwise \end{array}\right.,$$
$$\tau_{f(P)}(y_2) = \left\{\begin{array}{ll} \vee\{\tau_P(y_1): y_1 \in  f^{-1}(y_2)\} \\
\ \\
		
0,\; otherwise \end{array}\right.,$$ and
$$\eta_{f(P)}(y_2) = \left\{\begin{array}{ll} \wedge\{\eta_P(y_1): y_1 \in  f^{-1}(y_2)\} \\
\ \\
		
1,\; otherwise \end{array}\right.$$
Thus, $$f(P)(y_2) = \left\{\begin{array}{ll} (\vee\{\sigma_P(y_1): y_1 \in  f^{-1}(y_2)\}, \vee\{\tau_P(y_1): y_1 \in  f^{-1}(y_2)\}, \wedge\{\eta_P(y_1): y_1 \in  f^{-1}(y_2)\} )\\
\ \\
		
(0,0,1), \; \; \; \; \;  otherwise 
\end{array}\right.$$

The pre-image of $Q$ under $f,$ denoted by $f^{-1}(Q)$ is also defined as $$f^{-1}(Q)(y_1) = (\sigma_{f^{-1}(Q)}(y_1), \tau_{f^{-1}(Q)}(y_1), \eta_{f^{-1}(Q)}(y_1))$$ where $$\sigma_{f^{-1}(Q)}(y_1) = \sigma_Q(f(y_1)), \tau_{f^{-1}(Q)}(y_1) = \tau_Q(f(y_1)), \eta_{f^{-1}(Q)}(y_1) = \eta_Q(f(y_1)).$$ Thus, $$f^{-1}(Q)(y_1) = (\sigma_Q(f(y_1)), \tau_Q(f(y_1)), \eta_Q(f(y_1))).$$
\end{definition}

\begin{remark} \cite{t4}
For any $y_1 \in Y_1,$ we have $\sigma_{f(P)}(f(y_1)) \geq \sigma_P(y_1),~\tau_{f(P)}(f(y_1)) \geq \tau_P(y_1)$ and $\eta_{f(P)}(f(y_1)) \leq \eta_P(y_1).$
\end{remark}

\begin{theorem} \cite{t4}
Let $f: Y_1 \rightarrow Y_2$ be a mapping. Then,
\begin{itemize}
\item [(i)] $f(C_{r, s, t} (P)) \subseteq C_{r, s, t} (f(P)),~ \text{for~ all}~ P \in PFS(Y_1)$
\item [(ii)] $f^{-1}(C_{r, s, t} (Q)) = C_{r, s, t} (Q)(f^{-1}(Q)),~ \text{for~ all}~ Q \in PFS(Y_2).$
\end{itemize}
\end{theorem}

\begin{theorem} \label{2}
If $P = (\sigma_P, \tau_P, \eta_P )$ be a PFSG of $G.$ Then, $P$ is PFNSG of $G$ if and only if $C_{r, s, t} (P)$ is a PFNSG of $G,$ for all $0 \leq r, s, t \leq 1$ with $0 \leq r + s + t \leq 1.$
\end{theorem}

\begin{proof}
Suppose that $P$ is a PFNSG of $G$\\ 
$\iff~ aP = Pa$ for all $a \in G$ \\
$\iff C_{r, s, t} (aP) = C_{r, s, t} (Pa),$ for all $0 \leq r, s, t \leq 1$ with $0 \leq r + s + t \leq 1$\\
$\iff~ a * C_{r, s, t} (P) = C_{r, s, t} (P) * a,$ for all $a \in G$\\
$\iff~ C_{r, s, t} (P)$ is a PFNSG of $G.$ 
\end{proof}

\begin{definition}
Let $P = (\sigma_P, \tau_P, \eta_P)$ and $Q = (\sigma_Q, \tau_Q, \eta_Q)$ be PFSGs of group $G.$ Then, $P$ and $Q$ are said to be picture fuzzy conjugate subgroups (PFCSG) of $G$ if for some $a \in G$ $$\sigma_{P}(g) = \sigma_{Q}(a^{-1} g a),~ \tau_{P}(g) = \tau_{Q}(a^{-1} g a)~ \text{and}~ \eta_{P}(g) = \eta_{Q}(a^{-1} g a),~ \text{for~ every}~ g \in G.$$ 
\end{definition}

\ \\ \ \\

\section{Direct Product of Picture Fuzzy Subgroups}
\begin{definition}
Let $P = \lbrace \langle p, \sigma_P(p), \tau_P(p), \eta_P(p) \rangle~|~p~ \in Y_1 \rbrace$ and $Q = \lbrace \langle q, \sigma_Q(q), \tau_Q(q), \eta_Q(q) \rangle~|~q~ \in Y_2 \rbrace$ be two PFSs. Then, the Cartesian product of $P$ and $Q$ is the PFS $$P \times Q = \lbrace \langle(p, q), \sigma_{P \times Q}(p, q), \tau_{P \times Q}(p, q), \eta_{P \times Q}(p, q)\rangle~|~(p, q)~ \in Y_1 \times Y_2 \rbrace,$$ where $\sigma_{P \times Q}(p, q) = \bigwedge\{\sigma_P(p), \sigma_Q(q)\}$, $\tau_{P \times Q}(p, q) = \bigwedge\{ \tau_P(p), \tau_Q(q)\}$\\ and $\eta_{P \times Q}(p, q) = \bigvee\{ \eta_P(p) \vee \eta_Q(q)\}$ for all $(p, q) \in Y_1 \times Y_2.$
\end{definition}

\begin{theorem}
Let $P = (\sigma_P, \tau_P, \eta_P)$ and $Q = (\sigma_Q, \tau_Q, \eta_Q)$ be PFSs of $Y_1$ and $Y_2,$ respectively. Then, $$C_{r, s, t} (P \times Q) = C_{r, s, t} (P) \times C_{r, s, t} (Q),~\forall~ 0 \leq r, s, t \leq 1~\text{with}~r + s + t \in [0, 1].$$ 
\end{theorem}

\begin{proof}
Let $(p, q) \in C_{r, s, t} (P \times Q)$ be any element\\ 
$\iff~(p, q) \geq r,~(p, q) \geq s~\text{and}~(p, q) \leq t $\\ $\iff~\bigwedge \{\sigma_P(p),~\sigma_Q(q)\},~\bigwedge \{\tau_P(p),~\tau_Q(q)\}~\text{and}~\bigvee \{\eta_P(p),~\eta_Q(q)\} $\\ $\iff~\sigma_P(p) \geq r,~\sigma_Q(q) \geq r,~\tau_P(p) \geq s,~\tau_Q(q) \geq s~\text{and}~\eta_P(p) \leq t,~\eta_Q(q) \leq t$\\ $\iff~\sigma_P(p) \geq r,~ \tau_P(p) \geq s,~\eta_P(p) \leq t~ \text{and}~ \sigma_Q(q) \geq r,~ \tau_Q(q) \geq s,~\eta_Q(q) \leq t$\\ $\iff~p \in C_{r, s, t} (P)~ \text{and}~ q \in C_{r, s, t} (Q)$\\ $\iff~ (p, q) \in C_{r, s, t} (P) \times C_{r, s, t} (Q).$\\ Hence, $$C_{r, s, t} (P \times Q) = C_{r, s, t} (P) \times C_{r, s, t} (Q).$$ 
\end{proof}

\begin{theorem}
Let $P = (\sigma_P, \tau_P, \eta_P)$ and $Q = (\sigma_Q, \tau_Q, \eta_Q)$ be PFSGs of groups $G$ and $G^{*},$ respectively. Then, $P \times Q$ is also a PFSG of group $G \times G^{*}.$
\end{theorem}

\begin{proof}
Since $P$ and $Q$ are PFSG of $G$ and $G^{*},$ respectively, then $C_{r, s, t} (P)$ and $C_{r, s, t} (Q)$ are PFSG of $G$ and $G^{*},$ respectively, $\forall~r , s, t \in [0, 1]$ with $r + s + t \in [0, 1]$ by Theorem \ref{1}. This shows that $C_{r, s, t} (P) \times C_{r, s, t} (Q)$ is a PFSG of $G \times G^{*}$ which also implies that $C_{r, s, t} (P \times Q)$ is a PFSG of $G \times G^{*}$ by Theorem \ref{1}.    
\end{proof}

\begin{theorem} \label{a}
Let $P = (\sigma_P, \tau_P, \eta_P)$ and $Q = (\sigma_Q, \tau_Q, \eta_Q)$ be PFNSGs of groups $G$ and $G^{*},$ respectively. Then, $P \times Q$ is also a PFNSG of group $G \times G^{*}.$
\end{theorem}

\begin{proof}
Since $P$ and $Q$ are PFNSG of $G$ and $G^{*},$ respectively. Then, by Theorem \ref{2}, $C_{r, s, t} (P)$ and $C_{r, s, t} (Q)$ are PFNSG of $G$ and $G^{*},$ respectively. This means that $C_{r, s, t} (P) \times C_{r, s, t} (Q)$ is a PFNSG of $G \times G^{*}$ which also implies that $C_{r, s, t} (P \times Q)$ is a PFSG of $G \times G^{*}$ by Theorem \ref{2}. 
\end{proof}

\begin{theorem} \label{b}
Let $P = (\sigma_P, \tau_P, \eta_P)$ and $Q = (\sigma_Q, \tau_Q, \eta_Q)$ be PFSGs of groups $G$ and $G^{*},$ respectively. Assume that $e_1$ and $e_2$ are the identities of $G$ and $G^{*},$ respectively and if $P \times Q$ is a PFSG of $G \times G^{*},$ then at least on of the following must holds:\\
(i)~ $\sigma_Q(e_2) \geq \sigma_P(g),~ \tau_Q(e_2) \geq \tau_P(g),~\text{and}~\eta_Q(e_2) \leq \eta_P(g),~\forall~g \in G.$\\ 
(ii)~ $\sigma_Q(e_1) \geq \sigma_P(g^*),~ \tau_Q(e_1) \geq \tau_P(g^*),~\text{and}~\eta_Q(e_1) \leq \eta_P(g^*),~\forall~g^* \in G^*.$
\end{theorem}

\begin{proof}
Let $P \times Q$ be a PFSG of $G \times G^{*}.$ Suppose none of (i) and (ii) holds. Then, there exists $g \in G$ and $g^* \in G^*$ such that\\ $\sigma_P(g) > \sigma_Q(e_2),~ \tau_P(g) > \tau_Q(e_2), \eta_P(g) < \eta_Q(e_2)$ and \\ $\sigma_Q(g^*) > \sigma_P(e_1),~ \tau_Q(g^*) > \tau_P(e_1), \eta_Q(g^*) < \eta_P(e_1).$ So, $$\sigma_{P \times Q}(g, g^*) = \bigwedge \{\sigma_P(g), \sigma_Q(g^*)\} > \{\sigma_P(e_1), \sigma_Q(e_2)\} = \sigma_{P \times Q}(e_1, e_2),$$ $$\tau_{P \times Q}(g, g^*) = \bigwedge \{\tau_P(g), \tau_Q(g^*)\} > \{\tau_P(e_1), \tau_Q(e_2)\} = \tau_{P \times Q}(e_1, e_2)$$ and $$\eta_{P \times Q}(g, g^*) = \bigvee \{\eta_P(g), \eta_Q(g^*)\} < \{\eta_P(e_1), \eta_Q(e_2)\} = \eta_{P \times Q}(e_1, e_2).$$ This means that $P \times Q$ is not a PFSG of $G \times G^{*},$ which is a contradiction. Therefore, either $\sigma_Q(e_2) \geq \sigma_P(g),~ \tau_Q(e_2) \geq \tau_P(g),~\text{and}~\eta_Q(e_2) \leq \eta_P(g),~\forall~g \in G$ or  $\sigma_Q(e_1) \geq \sigma_P(g^*),~ \tau_Q(e_1) \geq \tau_P(g^*),~\text{and}~\eta_Q(e_1) \leq \eta_P(g^*),~\forall~g^* \in G^*.$ 
\end{proof}

\begin{theorem} \label{c}
Let $P = (\sigma_P, \tau_P, \eta_P)$ and $Q = (\sigma_Q, \tau_Q, \eta_Q)$ be PFSGs of groups $G$ and $G^{*},$ respectively such that $\sigma_P(g) \leq \sigma_Q(e_2),~ \tau_P(g) \leq \tau_Q(e_2)$ and $\eta_P(g) \geq \eta_Q(e_2)$ holds for all $g \in G,$ $e_2$ being the identity elements of $G.$ If $P \times Q$ is a PFSG of $G \times G^*,$ then $P$ is a PFSG of $G.$
\end{theorem}

\begin{proof}
Let $g, g^* \in G$ be any element. Then, $(g, e_2)~ \text{and}~ (g^*, e_2) \in G \times G^*.$ Since $\sigma_P(g) \leq \sigma_Q(e_2),~ \tau_P(g) \leq \tau_Q(e_2)$ and $\eta_P(g) \geq \eta_Q(e_2)$ holds for all $g \in G,$ we have, 

\begin{eqnarray*}
\sigma_P(g (g^*)^{-1}) & = & \bigwedge \lbrace \sigma_P(g (g^*)^{-1}), \sigma_Q(e_2, e_2) \rbrace\\
& = & \sigma_{P \times Q}(g (g^*)^{-1}, (e_2, e_2))\\
& = & \sigma_{P \times Q}[(g, e_2)((g^*)^{-1}, e_2)]\\
& \geq & \bigwedge \lbrace \sigma_{P \times Q}(g, e_2), \sigma_{P \times Q}((g^*)^{-1}, e_2) \rbrace\\
& = & \bigwedge \lbrace \bigwedge\lbrace \sigma_P(g), \sigma_Q(e_2)\rbrace, \bigwedge\lbrace \sigma_P(g^*)^{-1}, \sigma_Q(e_2) \rbrace \rbrace\\
& = & \bigwedge\lbrace \sigma_P(g), \sigma_P((g^*)^{-1} \rbrace\\
& \geq & \bigwedge\lbrace \sigma_P(g), \sigma_P(g^*) \rbrace. 
\end{eqnarray*}

\begin{eqnarray*}
\tau_P(g (g^*)^{-1}) & = & \bigwedge \lbrace \tau_P(g (g^*)^{-1}), \tau_Q(e_2, e_2) \rbrace\\
& = & \tau_{P \times Q}(g (g^*)^{-1}, (e_2, e_2))\\
& = & \tau_{P \times Q}[(g, e_2)((g^*)^{-1}, e_2)]\\
& \geq & \bigwedge \lbrace \tau_{P \times Q}(g, e_2), \tau_{P \times Q}((g^*)^{-1}, e_2) \rbrace\\
& = & \bigwedge \lbrace \bigwedge\lbrace \tau_P(g), \tau_Q(e_2)\rbrace, \bigwedge\lbrace \tau_P(g^*)^{-1}, \tau_Q(e_2) \rbrace \rbrace\\
& = & \bigwedge\lbrace \tau_P(g), \tau_P((g^*)^{-1} \rbrace\\
& \geq & \bigwedge\lbrace \tau_P(g), \tau_P(g^*) \rbrace. 
\end{eqnarray*}
and
\begin{eqnarray*}
\eta_P(g (g^*)^{-1}) & = & \bigvee \lbrace \eta_P(g (g^*)^{-1}), \eta_Q(e_2, e_2) \rbrace\\
& = & \eta_{P \times Q}(g (g^*)^{-1}, (e_2, e_2))\\
& = & \eta_{P \times Q}[(g, e_2)((g^*)^{-1}, e_2)]\\
& \leq & \bigvee \lbrace \eta_{P \times Q}(g, e_2), \eta_{P \times Q}((g^*)^{-1}, e_2) \rbrace\\
& = & \bigvee \lbrace \bigvee\lbrace \eta_P(g), \eta_Q(e_2)\rbrace, \bigvee\lbrace \eta_P(g^*)^{-1}, \eta_Q(e_2) \rbrace \rbrace\\
& = & \bigvee\lbrace \eta_P(g), \eta_P((g^*)^{-1} \rbrace\\
& \leq & \bigvee\lbrace \eta_P(g), \eta_P(g^*) \rbrace. 
\end{eqnarray*}
Therefore, $P$ is a PFSG of $G.$
\end{proof}

\begin{corollary}
Let $P = (\sigma_P, \tau_P, \eta_P)$ and $Q = (\sigma_Q, \tau_Q, \eta_Q)$ be PFSGs of groups $G$ and $G^{*},$ respectively such that $\sigma_Q(g^*) \leq \sigma_Q(e_1),~ \tau_Q(g^*) \leq \tau_Q(e_1)$ and $\eta_Q(g) \geq \eta_Q(e_1)$ holds for all $g^* \in G^*,$ $e_1$ being the identity elements of $G.$ If $P \times Q$ is a PFSG of $G \times G^*,$ then $Q$ is a PFSG of $G^*.$
\end{corollary}
\ \\
From the following Theorems, Theorem \ref{a}, Theorem \ref{b} and Theorem \ref{c}, we have the next corollary.

\begin{corollary}
Let $P = (\sigma_P, \tau_P, \eta_P)$ and $Q = (\sigma_Q, \tau_Q, \eta_Q)$ be PFSGs of groups $G$ and $G^{*},$ respectively. If $P \times Q$ is a PFSG of $G \times G^*,$ then either $P$ is a PFSG of $G$ or $Q$ is a PFSG of $G^*.$
\end{corollary}

\begin{theorem}
Let $P_1 = (\sigma_{P_1}, \tau_{P_1}, \eta_{P_1})$ and $P_2 = (\sigma_{P_2}, \tau_{P_2}, \eta_{P_2})$ be PFSGs of group $G$ and $Q_1 = (\sigma_{Q_1}, \tau_{Q_1}, \eta_{Q_1})$ and $Q_2 = (\sigma_{Q_2}, \tau_{Q_2}, \eta_{Q_2})$ be PFSGs of groups $G^{*},$ respectively such that $P_1$ and $P_2$ are PFCSG of $G$ and $Q_1$ and $Q_2$ are PFCSG of $G^*.$ Then, the PFSG $P_1 \times Q_1$ of $G \times G^*$ is conjugate to the PFSG $P_2 \times Q_2$ of $G \times G^*.$
\end{theorem}

\begin{proof}
Since $P_1$ and $P_2$ are PFCSGs of $G.$ Therefore, there exists $a \in G$ such that $$\sigma_{P_1}(a) = \sigma_{P_2}(a^{-1} g a),~ \tau_{P_1}(a) = \tau_{P_2}(a^{-1} g a)~ \text{and}~ \eta_{P_1}(a) = \eta_{P_2}(a^{-1} g a),~ \forall~ g \in G.$$ Similarly, since $Q_1$ and $Q_2$ are PFCSGs of $G^*.$ Therefore, there exists $b \in G^*$ such that $$\sigma_{Q_1}(b) = \sigma_{Q_2}(b^{-1} g^* b),~ \tau_{Q_1}(b) = \tau_{Q_2}(b^{-1} g^* b)~ \text{and}~ \eta_{Q_1}(b) = \eta_{Q_2}(b^{-1} g^* b),~ \forall~ g^* \in G^*.$$ Now,  

\begin{eqnarray*}
\sigma_{P_1 \times Q_1}(g, g^*) & = & \bigwedge \lbrace \sigma_{P_1}(g), \sigma_{Q_1}(g^*) \rbrace\\
& = & \bigwedge\lbrace \sigma_{P_2}(a^{-1}ga), \sigma_{Q_2}(b^{-1}g^*b) \rbrace\\
& = &  \sigma_{P_2 \times Q_2}((a^{-1}ga), (b^{-1}g^*b))\\
& = &  \sigma_{P_2 \times Q_2}((a^{-1}, b^{-1}) (g, g^*)(a, b)). 
\end{eqnarray*}

\begin{eqnarray*}
\tau_{P_1 \times Q_1}(g, g^*) & = & \bigwedge \lbrace \tau_{P_1}(g), \tau_{Q_1}(g^*) \rbrace\\
& = & \bigwedge\lbrace \tau_{P_2}(a^{-1}ga), \tau_{Q_2}(b^{-1}g^*b) \rbrace\\
& = &  \tau_{P_2 \times Q_2}((a^{-1}ga), (b^{-1}g^*b))\\
& = &  \tau_{P_2 \times Q_2}((a^{-1}, b^{-1}) (g, g^*)(a, b)). 
\end{eqnarray*}
and 
\begin{eqnarray*}
\eta_{P_1 \times Q_1}(g, g^*) & = & \bigvee \lbrace \eta_{P_1}(g), \eta_{Q_1}(g^*) \rbrace\\
& = & \bigvee\lbrace \eta_{P_2}(a^{-1}ga), \eta_{Q_2}(b^{-1}g^*b) \rbrace\\
& = &  \eta_{P_2 \times Q_2}((a^{-1}ga), (b^{-1}g^*b))\\
& = &  \eta_{P_2 \times Q_2}((a^{-1}, b^{-1}) (g, g^*)(a, b)). 
\end{eqnarray*}
Therefore, the PFSG $P_1 \times Q_1$ of $G \times G^*$ is conjugate to the PFSG $P_2 \times Q_2$ of $G \times G^*.$
\end{proof}

\ \\ \ \\

\section*{Conclusion}
In this paper, we have established some characterisations of direct product of picture fuzzy subgroups of a group via  $(r, s, t)$-cut sets of picture fuzzy sets. 
\ \\ \ \\ \ \\

%%%%%%%%%%%%%%%%%%%%%%%%%%%%%%%%%%%%%%%%%%%%%%%%%%%%%%%%%%%%%%

\end{document}